\documentclass[12pt]{amsart}

\newtheorem{theorem}[equation]{Theorem}
\newtheorem{lemma}[equation]{Lemma}
\newtheorem{corollary}[equation]{Corollary}
\newtheorem{proposition}[equation]{Proposition}
\newtheorem{hypothesis}[equation]{Hypothesis}

\numberwithin{equation}{section}

\usepackage{amscd,amsmath,amssymb,verbatim}

\begin{document}

\title[A generalization of the Hasse-Witt matrix]{A generalization of the Hasse-Witt matrix \\ of a hypersurface}
\author{Alan Adolphson}
\address{Department of Mathematics\\
Oklahoma State University\\
Stillwater, Oklahoma 74078}
\email{adolphs@math.okstate.edu}
\author{Steven Sperber}
\address{School of Mathematics\\
University of Minnesota\\
Minneapolis, Minnesota 55455}
\email{sperber@math.umn.edu}
\date{\today}
\begin{abstract}
The Hasse-Witt matrix of a hypersurface in ${\mathbb P}^n$ over a finite field of characteristic $p$ gives essentially complete mod $p$ information about the zeta function of the hypersurface.  But if the degree $d$ of the hypersurface is $\leq n$, the zeta function is trivial mod $p$ and the Hasse-Witt matrix is zero-by-zero.  We generalize a classical formula for the Hasse-Witt matrix to obtain a matrix that gives a nontrivial congruence for the zeta function for all $d$.  We also describe the differential equations satisfied by this matrix and prove that it is generically invertible. 
\end{abstract}
\maketitle

\section{Introduction}

Let $p$ be a prime number, let $q=p^a$, and let ${\mathbb F}_q$ be the field of $q$ elements.  Let
\begin{equation}
f_\lambda(x) = \sum_{j=1}^N \lambda_j x^{{\bf a}_j}\in {\mathbb F}_q[x_0,\dots,x_n]
\end{equation}
be a homogeneous polynomial of degree $d$.  We write ${\bf a}_j = (a_{0j},\dots,a_{nj})$ with $\sum_{i=0}^n a_{ij} = d$.  Let $X_\lambda\subseteq{\mathbb P}^n$ be the hypersurface defined by the equation $f_\lambda(x) = 0$ and let $Z(X_\lambda/{\mathbb F}_q,t)$ be its zeta function.  We write
\begin{equation}
Z(X_\lambda/{\mathbb F}_q,t) = \frac{P_\lambda(t)^{(-1)^n}}{(1-t)(1-qt)\cdots(1-q^{n-1}t)}
\end{equation}
for some rational function $P_\lambda(t)\in 1+t{\mathbb Z}[[t]]$.  If $d=1$ then $P_\lambda(t) = 1$, so we shall always assume that $d\geq 2$.

Define a nonnegative integer $\mu$ by the equation
\[ \bigg\lceil \frac{n+1}{d}\bigg\rceil = \mu+1, \]
where $\lceil r\rceil$ denotes the least integer greater than or equal to the real number $r$.  By a result of Ax\cite{A} (see also Katz\cite[Proposition~2.4]{K}) we have
\[ P_\lambda(q^{-\mu}t) \in 1+t{\mathbb Z}[[t]]. \]
Our goal in this paper is to give a mod $p$ congruence for $P_\lambda(q^{-\mu}t)$.  We do this by defining a generalization of the classical Hasse-Witt matrix, which gives such a congruence for $\mu=0$.  Presumably our matrix is the matrix of a ``higher Hasse-Witt'' operation as defined by Katz\cite[Section 2.3.4]{K2}, but so far we have not been able to prove this.

It will be convenient to define an augmentation of the vectors ${\bf a}_j$.  Set
\[ {\bf a}_j^+ = (a_{0j},\dots,a_{nj},1)\in{\mathbb N}^{n+2},\quad j=1,\dots,N, \]
where ${\mathbb N}$ denotes the nonnegative integers.  Note that the vectors ${\bf a}_j^+$ all lie on the hyperplane $\sum_{i=0}^n u_i=du_{n+1}$ in ${\mathbb R}^{n+2}$.  We shall be interested in the lattice points on this hyperplane that lie in $({\mathbb R}_{>0})^{n+2}$: set
\[ U =  \bigg\{ u=(u_0,\dots,u_{n+1})\in{\mathbb N}^{n+2}\mid \text{$\sum_{i=0}^n u_i = du_{n+1}$ and  $u_i>0$ for all $i$}\bigg\}. \]
Note that $u\in U$ implies that $u_{n+1}\geq\mu+1$.  Let
\[ U_{\min} = \{ u=(u_0,\dots,u_{n+1})\in U\mid u_{n+1} = \mu+1\}, \]
a nonempty set by the definition of $\mu$.  We define a matrix of polynomials with rows and columns indexed by $U_{\min}$: let $A(\Lambda) = [A_{uv}(\Lambda)]_{u,v\in U_{\min}}$, where
\[ A_{uv}(\Lambda) = (-1)^{\mu+1}\sum_{\substack{\nu\in{\mathbb N}^N\\ \sum_{j=1}^N \nu_j{\bf a}_j^+ = pu-v}} \frac{\Lambda_1^{\nu_1}\cdots\Lambda_N^{\nu_N}}{\nu_1!\cdots\nu_N!}\in {\mathbb Q}[\Lambda_1,\dots,\Lambda_N]. \]

Note that since the $(n+1)$-st coordinate of each ${\bf a}_j^+$ equals 1, the condition on the summation implies that
\[ \sum_{j=1}^N \nu_j = (p-1)(\mu+1). \]
When $\mu=0$, it follows that $\nu_j\leq p-1$ for all $j$, hence the matrix $A(\Lambda)$ can be reduced modulo $p$.  We denote by $\bar{A}(\Lambda)\in{\mathbb F}_p[\Lambda]$ its reduction modulo $p$.  Using the algorithm of Katz\cite[Algorithm 2.3.7.14]{K2}, one then checks that $\bar{A}(\lambda)$ is the Hasse-Witt matrix of the hypersurface $f_\lambda =0$.  It is somewhat surprising that even when $\mu>0$ we still have $\nu_j\leq p-1$ for all $j$.  

\begin{lemma}
If $u,v\in U_{\min}$, $\nu\in{\mathbb N}^N$, and $\sum_{j=1}^N \nu_j{\bf a}_j^+ = pu-v$, then $\nu_j\leq p-1$ for all $j$.  In particular, $A_{uv}(\Lambda)\in({\mathbb Q}\cap{\mathbb Z}_p)[\Lambda]$, so $A_{uv}(\Lambda)$ can be reduced modulo $p$.
\end{lemma}

The proof of Lemma 1.3 will be given in Section 2.  By the results of \cite[Theorem~2.7 or Theorem~3.1]{AS}, which will be recalled in Section 2, Lemma~1.3 implies immediately that each $A_{uv}(\Lambda)$ is a mod $p$ solution of an $A$-hypergeometric system of differential equations.  

Write the rational function $P_\lambda(t)$ of (1.2) as 
\[ P_\lambda(t) = \frac{Q_\lambda(t)}{R_\lambda(t)}, \]
where $Q_\lambda(t)$ and $R_\lambda(t)$ are relatively prime polynomials with 
\[ Q_\lambda(q^{-\mu}t),\;R_\lambda(q^{-\mu}t) \in 1+t{\mathbb Z}[t]. \]
If $X_\lambda$ is smooth, it is known that $P_\lambda(t)$ is a polynomial, i.~e., $R_\lambda(t)=1$.  
Our main result is the following, which does not require any smoothness assumption.

\begin{theorem}
If $n$ is not divisible by $d$, then $R_\lambda(q^{-\mu}t)\equiv 1\pmod{q}$ and 
\[ Q_\lambda(q^{-\mu}t) \equiv \det\big(I-t\bar{A}(\lambda^{p^{a-1}})\bar{A}(\lambda^{p^{a-2}})\cdots \bar{A}(\lambda)\big)\pmod{p}. \]
\end{theorem}

Note that even in the classical case of the Hasse-Witt matrix ($\mu=0$), this result contains something new, as we do not assume that $X_\lambda$ is a smooth hypersurface.

The proof of Theorem~1.4 will occupy Sections~\mbox{3--5}.  To describe the zeta function, we apply the $p$-adic cohomology theory of Dwork, as in Katz\cite[Sections 4--6]{K3}.  Indeed, Equation~(3.5) below is a refined version of \cite[Equation~(4.5.33)]{K3}.  We discuss the case $d\,|\,n$ in Section~6.  If $d\,|\, n$, the conclusion of Theorem 1.4 need not hold, and the rational function $P_\lambda(q^{-\mu}t)\pmod{p}$ is instead described by Theorem 6.2.  We prove the generic invertibility of the matrix $\bar{A}(\Lambda)$ in Section 7.

\section{Proof of Lemma 1.3}

It will be convenient for later applications to prove a more general version of Lemma 1.3.  Put $S=\{0,1,\dots,n\}$ and let $I\subseteq S$.  Define an integer $\mu_I$ by the equation
\[ \bigg\lceil\frac{|I|}{d}\bigg\rceil = \mu_I + 1. \]
Note that $\mu_I\geq 0$ if $I\neq\emptyset$, $\mu_\emptyset = -1$, and, in the notation of the Introduction, $\mu_S = \mu$.  Set
\[ U^I =  \bigg\{ u=(u_0,\dots,u_{n+1})\in{\mathbb N}^{n+2}\mid \text{$\sum_{i=0}^n u_i = du_{n+1}$ and  $u_i>0$ for all $i\in I$}\bigg\}. \]
Note that $u\in U^I$ implies that $u_{n+1}\geq\mu_I+1$.  Let
\[ U^I_{\min} = \{ u=(u_0,\dots,u_{n+1})\in U^I\mid u_{n+1} = \mu_I+1\}, \]
a nonempty set by the definition of $\mu_I$.  Lemma 1.3 is the special case $I=S$ of the following result.

\begin{lemma}
If $u,v\in U^I_{\min}$, $\nu\in{\mathbb N}^N$, and $\sum_{j=1}^N \nu_j{\bf a}_j^+ = pu-v$, then $\nu_j\leq p-1$ for all $j$. 
\end{lemma}

\begin{proof}
The result is trivial when $I=\emptyset$ since $U^\emptyset_{\min} = \{(0,\dots,0)\}$, so assume $I\neq\emptyset$.  
Let  $u=(u_0,\dots,u_n,\mu_I+1),v=(v_0,\dots,v_n,\mu_I+1)\in U^I_{\min}$.  Fix $k\in\{1,\dots,N\}$.  We claim there exists an index $i_0\in\{0,\dots,n\}$ such that
\begin{equation}
a_{i_0k}\geq \begin{cases} u_{i_0} & \text{if $i_0\in I$,} \\ u_{i_0}+1 & \text{if $i_0\not\in I$.} \end{cases}
\end{equation}
For if (2.2) fails for all $i_0\in\{0,\dots,n\}$, then
\[ u-{\bf a}_k^+ = (u_0-a_{0k},\dots,u_n-a_{nk},\mu_I)\in U^I, \]
contradicting the definition of $\mu_I$.

If $\nu_k\geq p$, then
\[ \nu_ka_{i_0k}\geq pa_{i_0k}\geq \begin{cases} pu_{i_0} & \text{if $i_0\in I$,} \\ pu_{i_0}+p & \text{if $i_0\not\in I$,} \end{cases} \]
hence in both cases we have
\[ \nu_ka_{i_0k}>pu_{i_0}-v_{i_0}. \]
But our hypothesis $\sum_{j=1}^N \nu_j{\bf a}_j^+ = pu-v$ implies that
\[ \nu_ka_{i_0k}\leq pu_{i_0}-v_{i_0}. \]
This contradiction shows that $\nu_k\leq p-1$.  And since $k$ was arbitrary, the lemma is established.
\end{proof}

We recall the definition of the $A$-hypergeometric system of differential equations associated to the set $A = \{{\bf a}_j^+\}_{j=1}^N$.  Let $L\subseteq {\mathbb Z}^N$ be the lattice of relations on $A$,
\[ L = \bigg\{ l=(l_1,\dots,l_N)\in{\mathbb Z}^N\mid \sum_{j=1}^N l_j{\bf a}_j^+ = {\bf 0}\bigg\}, \]
and let $\beta=(\beta_0,\dots,\beta_{n+1})\in {\mathbb C}^{n+2}$.  The $A$-hypergeometric system with parameter~$\beta$ is the system of partial differential operators in variables $\Lambda_1,\dots,\Lambda_N$ consisting of the box operators
\[ \Box_l = \prod_{l_j>0} \bigg(\frac{\partial}{\partial\Lambda_j}\bigg)^{l_j} - \prod_{l_j<0} \bigg(\frac{\partial}{\partial\Lambda_j}\bigg)^{-l_j} \quad\text{for $l\in L$} \]
and the Euler (or homogeneity) operators
\[ Z_i = \sum_{j=1}^N a_{ij}\Lambda_j\frac{\partial}{\partial\Lambda_j} - \beta_i\quad\text{for $i=0,\dots,n$} \]
and
\[ Z_{n+1} = \sum_{j=1}^N \Lambda_j\frac{\partial}{\partial\Lambda_j} - \beta_{n+1}. \]

Let $A^I(\Lambda) = [A^I_{uv}(\Lambda)]_{u,v\in U^I_{\min}}$, where
\begin{equation}
A^I_{uv}(\Lambda) = (-1)^{\mu_I+1}\sum_{\substack{\nu\in{\mathbb N}^N\\ \sum_{j=1}^N \nu_j{\bf a}_j^+ = pu-v}} \frac{\Lambda_1^{\nu_1}\cdots\Lambda_N^{\nu_N}}{\nu_1!\cdots\nu_N!}\in {\mathbb Q}[\Lambda_1,\dots,\Lambda_N]. 
\end{equation}
Note that in the notation of the Introduction we have $A^S_{uv}(\Lambda) = A_{uv}(\Lambda)$.  
By Lemma 2.1, the polynomials $A^I_{uv}(\Lambda)$ have $p$-integral coefficients.  
Lemma 2.1 also says that $pu-v$ is very good in the sense of \cite[Section~2]{AS}.  We may therefore apply  \cite[Theorem~2.7]{AS} (or \cite[Theorem 3.1]{AS} since this system is nonconfluent) to conclude that $\bar{A}^I_{uv}(\Lambda)$ is a mod $p$ solution of the $A$-hypergeometric system with parameter $\beta = pu-v$ (or, equivalently, $\beta=-v$ since we have reduced modulo $p$).  

\section{The zeta function}

To make a connection between the matrix $A(\Lambda)$ and the zeta function~(1.2), we apply a consequence of the Dwork trace formula developed in \cite{AS2} (see Equation~(3.5) below).  Let $\gamma_0$ be a zero of the series $\sum_{i=0}^\infty t^{p^i}/p^i$ having ${\rm ord}_p\:\gamma_0 = 1/(p-1)$, where ${\rm ord}_p$ is the $p$-adic valuation normalized by ${\rm ord}_p\:p = 1$.  Let $L_0$ be the space of series
\[ L_0 =  \bigg\{\sum_{u \in{\mathbb N}^{n+2}} c_u\gamma_0^{pu_{n+1}}x^u\mid\text{$\sum_{i=0}^n u_i-du_{n+1}=0$, $c_u\in{\mathbb C}_p$, and $\{c_u\}$ is bounded}\bigg\}. \]  
For $I\subseteq\{0,\dots,n\}$, let $L_0^I$ be the subset of $L_0$ defined by
\[ L_0^I = \bigg\{ \sum_{u \in{\mathbb N}^{n+2}} c_u\gamma_0^{pu_{n+1}}x^u\in L_0\mid \text{$u_i>0$ for $i\in I$}\bigg\}. \]
Let $ {\rm AH}(t)= \exp(\sum_{i=0}^{\infty}t^{p^i}/p^i)$ be the Artin-Hasse series, a power series in $t$ that has $p$-integral coefficients, and set 
\[ \theta(t) = {\rm AH}({\gamma}_0t)=\sum_{i=0}^{\infty}\theta_i t^i. \]
We then have
\begin{equation}
{\rm ord}_p\: \theta_i\geq \frac{i}{p-1}.
\end{equation}

We define the Frobenius operator on $L_0$.  Put
\begin{equation}
\theta(\hat{\lambda},x) = \prod_{j=1}^N \theta(\hat{\lambda}_jx^{{\bf a}^+_j}),
\end{equation}
where $\hat{\lambda}$ denotes the Teichm\"uller lifting of $\lambda$.  
We shall also need to consider the series $\theta_0(\hat{\lambda},x)$ defined by
\begin{equation}
\theta_0(\hat{\lambda},x) = \prod_{i=0}^{a-1} \prod_{j=1}^N \theta\big((\hat{\lambda}_jx^{{\bf a}^+_j})^{p^i}\big) = \prod_{i=0}^{a-1}\theta(\hat{\lambda}^{p^i},x^{p^i}).
\end{equation}
Define an operator $\psi$ on formal power series by
\begin{equation}
\psi\bigg(\sum_{u\in{\mathbb N}^{n+2}} c_ux^u\bigg) = \sum_{u\in{\mathbb N}^{n+2}} c_{pu}x^u.
\end{equation}
Denote by $\alpha_{\hat{\lambda}}$ the composition
\[ \alpha_{{\hat{\lambda}}} := \psi^a\circ\text{``multiplication by $\theta_0(\hat{\lambda},x)$.''} \]
The map $\alpha_{\hat{\lambda}}$ operates on $L_0$ and is stable on each $L_0^I$.  The proof of Theorem 1.4 will be based on the following formula for the rational function $P_\lambda(t)$ defined in~(1.2).  By \cite[Equation 7.12]{AS2} we have
\begin{equation}
P_\lambda(qt) = \prod_{I\subseteq\{0,1,\dots,n\}} 
\det(I-q^{n+1-|I|}t\alpha_{\hat{\lambda}}\mid L_0^I)^{(-1)^{n+1+|I|}}.
\end{equation}

To exploit (3.5) we shall need $p$-adic estimates for the action of $\alpha_{\hat{\lambda}}$ on $L^I_0$.  
Expand (3.3) as a series in $x$, say,
\begin{equation}
\theta_0(\hat{\lambda},x) = \sum_{w\in{\mathbb N}A} \theta_{0,w}(\hat{\lambda})x^w.
\end{equation}
Note that from the definitions we have $\theta_{0,w}(\hat{\lambda})\in{\mathbb Q}_p(\zeta_{q-1},\gamma_0)$.  
A direct calculation shows that for $v\in U^I$,
\begin{equation}
\alpha_{\hat{\lambda}}(x^v) =\sum_{u\in U^I} \theta_{0,qu-v}(\hat{\lambda})x^u,
\end{equation}
thus we need $p$-adic estimates for the $\theta_{0,qu-v}(\hat{\lambda})$ with $u,v\in U^I$.

Expand (3.2) as a series in $x$:
\begin{equation}
\theta(\hat{\lambda},x) = \sum_{w\in{\mathbb N}A} \theta_w(\hat{\lambda})x^w,
\end{equation}
where 
\begin{equation}
\theta_w(\hat{\lambda}) = \sum_{\nu\in{\mathbb N}^N} \theta_\nu^{(w)}\hat{\lambda}^\nu
\end{equation}
with
\begin{equation}
\theta_\nu^{(w)} = \begin{cases} \prod_{j=1}^N \theta_{\nu_j} & \text{if $\sum_{j=1}^N \nu_j{\bf a}_j^+ = w$,} \\ 0 & \text{if $\sum_{j=1}^N \nu_j{\bf a}_j^+ \neq w$.} \end{cases}
\end{equation}
From (3.1) we have the estimate 
\begin{equation}
{\rm ord}_p\:\theta_\nu^{(w)}\geq\frac{\sum_{j=1}^N\nu_j}{p-1} = \frac{w_{n+1}}{p-1}.
\end{equation}
In particular, this implies the estimate
\begin{equation}
{\rm ord}_p\:\theta_w(\hat{\lambda}) \geq \frac{w_{n+1}}{p-1}.
\end{equation}

By (3.3) and (3.8) we have
\begin{equation}
\theta_{0,w}(\hat{\lambda}) = \sum_{\substack{u^{(0)},\dots,u^{(a-1)}\in{\mathbb N}A\\ \sum_{i=0}^{a-1}
p^iu^{(i)}=w}} \prod_{i=0}^{a-1} \theta_{u^{(i)}}(\hat{\lambda}^{p^i}).
\end{equation}
In particular, we get the formula
\begin{equation}
\theta_{0,qu-v}(\hat{\lambda}) = \sum_{\substack{w^{(0)},\dots,w^{(a-1)}\in{\mathbb N}A\\ \sum_{i=0}^{a-1}
p^iw^{(i)}=qu-v}} \prod_{i=0}^{a-1} \theta_{w^{(i)}}(\hat{\lambda}^{p^i}). 
\end{equation}
Applying (3.12) to the products on the right-hand side of (3.14) gives
\begin{equation}
{\rm ord}_p\:\bigg(\prod_{i=1}^{a-1} \theta_{w^{(i)}}(\hat{\lambda}^{p^i})\bigg)\geq \sum_{i=0}^{a-1} \frac{w^{(i)}_{n+1}}{p-1}.
\end{equation}
This estimate is not directly helpful for estimating $\theta_{0,qu-v}(\hat{\lambda})$ because we lack information about the $w^{(i)}$.  Instead we proceed as follows.  

Fix $w^{(0)},\dots,w^{(a-1)}\in{\mathbb N}A$ with 
\begin{equation}
\sum_{i=0}^{a-1} p^iw^{(i)} = qu-v.
\end{equation}  
We construct inductively from $\{w^{(i)}\}_{i=0}^{a-1}$ a related sequence $\{\tilde{w}^{(i)}\}_{i=0}^a\subseteq U^I$ such that 
\begin{equation}
w^{(i)} = p\tilde{w}^{(i+1)}-\tilde{w}^{(i)}\quad\text{for $i=0,\dots,a-1$.}
\end{equation}
First of all, take $\tilde{w}^{(0)} = v$.  Eq.~(3.16) shows that $w^{(0)}+\tilde{w}^{(0)}=p\tilde{w}^{(1)}$ for some $\tilde{w}^{(1)}\in{\mathbb Z}^{n+2}$; since $w^{(0)}\in{\mathbb N}A$ and $\tilde{w}^{(0)}\in U^I$ we conclude that $\tilde{w}^{(1)}\in U^I$.  Suppose that for some $0< k\leq a-1$ we have defined $\tilde{w}^{(0)},\dots,\tilde{w}^{(k)}\in U^I$ satisfying (3.17) for $i=0,\dots,k-1$.  Substituting $p\tilde{w}^{(i+1)}-\tilde{w}^{(i)}$ for $w^{(i)}$ for $i=0,\dots,k-1$ in (3.16) gives
\begin{equation}
-\tilde{w}^{(0)}+ p^k\tilde{w}^{(k)} + \sum_{i=k}^{a-1} p^iw^{(i)} = p^au-v.
\end{equation}
Since $\tilde{w}^{(0)}=v$, we can divide this equation by $p^k$ to get $\tilde{w}^{(k)} + w^{(k)} = p\tilde{w}^{(k+1)}$ for some $\tilde{w}^{(k+1)}\in{\mathbb Z}^{n+2}$.  Since $w^{(k)}\in{\mathbb N}A$ and (by induction) $\tilde{w}^{(k)}\in U^I$, we conclude that $\tilde{w}^{(k+1)}\in U^I$.  This completes the inductive construction.  Note that in the special case $k=a-1$, this computation gives $\tilde{w}^{(a)} = u$.

Summing Eq.~(3.17) over $i=0,\dots,a-1$ and using $\tilde{w}^{(0)} = v$,  $\tilde{w}^{(a)} = u$, gives
\begin{equation}
\sum_{i=0}^{a-1} w^{(i)} = pu-v + (p-1)\sum_{i=1}^{a-1} \tilde{w}^{(i)},
\end{equation}
hence
\begin{equation}
\sum_{i=0}^{a-1} \frac{w^{(i)}_{n+1}}{p-1} = \frac{pu_{n+1}-v_{n+1}}{p-1} + \sum_{i=1}^{a-1} \tilde{w}^{(i)}_{n+1}.
\end{equation}
For $w^{(0)},\dots,w^{(a-1)}$ as in (3.16), we thus get from (3.15)
\begin{equation}
{\rm ord}_p\:\bigg(\prod_{i=0}^{a-1} \theta_{w^{(i)}}(\hat{\lambda}^{p^i})\bigg)\geq \frac{pu_{n+1}-v_{n+1}}{p-1} + \sum_{i=1}^{a-1} \tilde{w}^{(i)}_{n+1}.
\end{equation}
Since $\tilde{w}^{(i)}\in U^I$, we have
\begin{equation}
\text{$\tilde{w}^{(i)}_{n+1} =\mu_I+1$ if $\tilde{w}^{(i)}\in U^I_{\min}$ and $\tilde{w}^{(i)}_{n+1}  \geq\mu_I+2$ if $\tilde{w}^{(i)}\not\in U^I_{\min}$.} 
\end{equation}
From (3.21) and (3.22) we get the following result.
\begin{lemma}
For $u,v\in U^I$ and $w^{(0)},\dots,w^{(a-1)}$ as in $(3.16)$, we have
\begin{equation}
{\rm ord}_p\:\bigg(\prod_{i=0}^{a-1} \theta_{w^{(i)}}(\hat{\lambda}^{p^i})\bigg)\geq \frac{pu_{n+1}-v_{n+1}}{p-1} + (a-1)(\mu_I+1).
\end{equation}
Furthermore, if any of the terms $\tilde{w}^{(1)},\dots,\tilde{w}^{(a-1)}$ of the associated sequence satisfying $(3.17)$ is not contained in $U^I_{\min}$, then
\begin{equation}
{\rm ord}_p\:\bigg(\prod_{i=0}^{a-1} \theta_{w^{(i)}}(\hat{\lambda}^{p^i})\bigg)\geq \frac{pu_{n+1}-v_{n+1}}{p-1} + (a-1)(\mu_I+1)+1.
\end{equation}
\end{lemma}

Our desired estimate for $\theta_{0,qu-v}(\hat{\lambda})$ now follows from (3.14).
\begin{corollary}
For $u,v\in U^I$ we have
\begin{equation}
{\rm ord}_p\: \theta_{0,qu-v}(\hat{\lambda})\geq \frac{pu_{n+1}-v_{n+1}}{p-1} + (a-1)(\mu_I+1). 
\end{equation}
\end{corollary}

\section{The action of $\alpha_{\hat{\lambda}}$ on $L_0^I$}

In this section, we use Corollary 3.26 to study the action of $\alpha_{\hat{\lambda}}$ on~$L_0^I$.  
From (3.7) and the formula of Serre\cite[Proposition~7]{S} we have
\begin{equation}
\det(I-t\alpha_{\hat{\lambda}}\mid L_0^I) = \sum_{m=0}^\infty a^I_mt^m, 
\end{equation}
where
\begin{equation}
a^I_m = (-1)^m \sum_{U_m\subseteq U^I} \sum_{\sigma\in{\mathcal S}_m} {\rm sgn}(\sigma)\prod_{u\in U_m} \theta_{0,qu-\sigma(u)}(\hat{\lambda})\in {\mathbb Q}_p(\zeta_{q-1},\gamma_0),
\end{equation}
the outer sum is over all subsets $U_m\subseteq U^I$ of cardinality $m$, and ${\mathcal S}_m$ is the group of permutations on $m$ objects.  
\begin{proposition}
The coefficient $a^I_m$  is divisible by $q^{m(\mu_I+1)}$ and satisfies the congruence
\begin{equation}
 a^I_m\equiv (-1)^m \sum_{U_m\subseteq U_{\min}^I} \sum_{\sigma\in{\mathcal S}_m} {\rm sgn}(\sigma)\prod_{u\in U_m} \theta_{0,qu-\sigma(u)}(\hat{\lambda}) \pmod{pq^{m(\mu_I+1)}}.
\end{equation}
In particular, $a^I_m\equiv 0\pmod{pq^{m(\mu_I+1)}}$ if $m>|U^I_{\min}|$.
\end{proposition}

\begin{proof}
If $U_m\subseteq U^I$ is a subset of cardinality $m$ and $\sigma$ is a permutation of $U_m$, then by (3.27)
\begin{align*}
{\rm ord}_p\bigg(\prod_{u\in U_m}\theta_{0,qu-\sigma(u)}(\hat{\lambda})\bigg) &\geq m(a-1)(\mu_I+1) + \sum_{u\in U_m} \frac{pu_{n+1}-\sigma(u)_{n+1}}{p-1} \\
 &= m(a-1)(\mu_I+1) + \sum_{u\in U_m} u_{n+1} \\
 &\geq ma(\mu_I+1)
\end{align*}
since $u\in U_m$ implies $u_{n+1}\geq \mu_I+1$.  It follows from (4.2) that $a^I_m$ is divisible by~$q^{m(\mu_I+1)}$.  Furthermore, $u_{n+1}\geq\mu_I+2$ if $u\not\in U^I_{\min}$, so 
\begin{equation*}
{\rm ord}_p\bigg(\prod_{u\in U_m}\theta_{0,qu-\sigma(u)}(\hat{\lambda})\bigg) \geq ma(\mu_I+1)+1\quad\text{if $U_m\not\subseteq U^I_{\min}$.}
\end{equation*}
The congruence (4.4) now follows from (4.2).
\end{proof}

As an immediate corollary of Proposition 4.3, we have the following result.
\begin{corollary}
The reciprocal roots of $\det(I-t\alpha_{\hat{\lambda}}\mid L_0^I)$ are all divisible by $q^{\mu_I+1}$.  
\end{corollary}

Corollary 4.5 allows us to analyze the terms on the right-hand side of (3.5).
\begin{proposition}
The reciprocal roots of $\det(I-q^{n+1-|I|}t\alpha_{\hat{\lambda}}\mid L_0^I)$ are divisible by $q^{\mu+2}$ unless either $|I|=n+1$ or $|I|=n$ and $n$ is divisible by $d$, in which case they are divisible by $q^{\mu+1}$.  
\end{proposition}

\begin{proof}
Corollary 4.5 and the definition of $\mu_I$ imply that the reciprocal roots of $\det(I-q^{n+1-|I|}t\alpha_{\hat{\lambda}}\mid L_0^I)$ are divisible by $q$ to the power
\begin{equation}
n+1-|I| + \bigg\lceil \frac{|I|}{d}\bigg\rceil.
\end{equation}
If $|I|=n+1$, this reduces to $\mu+1$.  Suppose $|I|=n$.  From the definition of $\mu$ we have $n=\mu d + r$ with $0\leq r\leq d-1$.  The expression (4.7) then reduces to $1+\lceil (\mu d+r)/d\rceil$, which equals $\mu+2$ if $r>0$ and equals $\mu+1$ if $r=0$.  

If $|I|=n-1$, expression (4.7) reduces to $2+\lceil (\mu d+r-1)/d\rceil$, which equals $\mu+3$ if $r>1$ and equals $\mu+2$ if $r=0,1$.  Finally, note that expression (4.7) cannot decrease when $|I|$ decreases, so expression (4.7) will be $\geq \mu+2$ for $|I|<n-1$.
\end{proof}

From (3.5) and Proposition 4.6 we get the following result.
\begin{proposition}
If $n$ is not divisible by $d$, then
\begin{equation}
P_{\lambda}(q^{-\mu}t)\equiv\det(I-q^{-\mu-1}t\alpha_{\hat{\lambda}}\mid L_0^S) \pmod{q}.
\end{equation}
If $n$ is divisible by $d$, then
\begin{equation}
P_{\lambda}(q^{-\mu}t)\equiv \frac{\det(I-q^{-\mu-1}t\alpha_{\hat{\lambda}}\mid L_0^S)}{\prod_{i=0}^n \det(I-q^{-\mu}t\alpha_{\hat{\lambda}}\mid L_0^{S\setminus\{i\}})}\pmod{q}.
\end{equation}
\end{proposition}

\section{Proof of Theorem 1.4}

It follows from (4.9) that $R_\lambda(q^{-\mu}t)\equiv 1\pmod{q}$ if $n$ is not divisible by $d$.  To establish Theorem~1.4, it remains to prove the congruence for $Q_\lambda(q^{-\mu}t)$.

Consider the matrix $B^I(\hat{\lambda}) = [B^I_{uv}(\hat{\lambda})]_{u,v\in U^I_{\min}}$ defined by
\[ B^I_{uv}(\hat{\lambda}) = \theta_{0,qu-v}(\hat{\lambda}). \]
We have
\begin{equation}
\det\big(I-tB^I(\hat{\lambda})\big) = \sum_{m=0}^{|U^I_{\min}|} b_m^I t^m,
\end{equation}
where
\begin{equation}
b^I_m = (-1)^m \sum_{U_m\subseteq U^I_{\min}} \sum_{\sigma\in{\mathcal S}_m} {\rm sgn}(\sigma) \prod_{u\in U_m} \theta_{0,qu-\sigma(u)}(\hat{\lambda}).
\end{equation}

Proposition 4.3 implies $b_m^I\equiv 0\pmod{q^{m(\mu_I+1)}}$ and 
\[ b_m^I\equiv a_m^I\pmod{pq^{m(\mu_I+1)}}, \] 
which establishes the following result.
\begin{proposition} 
We have
\[ \det(I-q^{-\mu_I-1}t\alpha_{\hat{\lambda}}\mid L^I_0) \equiv \det(I-q^{-\mu_I-1}tB^I(\hat{\lambda}))\pmod{p}. \]
\end{proposition}

Combined with Proposition 4.8, this gives the following congruences.
\begin{corollary}
If $n$ is not divisible by $d$, then
\begin{equation}
P_{\lambda}(q^{-\mu}t)\equiv\det\big(I-q^{-\mu-1}tB^S(\hat{\lambda})\big) \pmod{p}.
\end{equation}
If $n$ is divisible by $d$, then
\begin{equation}
P_{\lambda}(q^{-\mu}t)\equiv \frac{\det\big(I-q^{-\mu-1}tB^S(\hat{\lambda})\big)}{\prod_{i=0}^n \det\big(I-q^{-\mu}tB^{S\setminus\{i\}}(\hat{\lambda})\big)}\pmod{p}.
\end{equation}
\end{corollary}

To simplify (5.5) and (5.6) further, we restate Lemma 3.23 in the special case where $u,v\in U^I_{\min}$, i.e., $u_{n+1} = v_{n+1} =\mu_I+1$.  
\begin{lemma}
For $u,v\in U^I_{\min}$ and $w^{(0)},\dots,w^{(a-1)}$ as in $(3.16)$, we have
\begin{equation}
{\rm ord}_p\:\bigg(\prod_{i=0}^{a-1} \theta_{w^{(i)}}(\hat{\lambda}^{p^i})\bigg)\geq a(\mu_I+1).
\end{equation}
Furthermore, if any of the terms $\tilde{w}^{(1)},\dots,\tilde{w}^{(a-1)}$ of the associated sequence satisfying $(3.17)$ is not contained in $U^I_{\min}$, then
\begin{equation}
{\rm ord}_p\:\bigg(\prod_{i=0}^{a-1} \theta_{w^{(i)}}(\hat{\lambda}^{p^i})\bigg)\geq a(\mu_I+1)+1.
\end{equation}
\end{lemma}

Applying Lemma 5.7 to Equations (3.14) and (3.17) gives the following congruence: for $u,v\in  U^I_{\min}$, 
\begin{multline}
\theta_{0,qu-v}(\hat{\lambda}) \equiv \\ 
\sum_{\substack{\tilde{w}^{(1)},\dots,\tilde{w}^{(a-1)}\in U_{I,\min}\\ \sum_{i=0}^{a-1} p^i(p\tilde{w}^{(i+1)}-\tilde{w}^{(i)}) = qu-v}} \prod_{i=0}^{a-1}\theta_{p\tilde{w}^{(i+1)}-\tilde{w}^{(i)}}(\hat{\lambda}^{p^i}) \pmod{pq^{\mu_I+1}},
\end{multline}
where $\tilde{w}^{(0)} = v$ and $\tilde{w}^{(a)} = u$.  

Let $C^I(\hat{\lambda}) = [C^I_{uv}(\hat{\lambda})]_{u,v\in U_{I,\min}}$ be the matrix 
\[ C^I_{uv}(\hat{\lambda}) = \theta_{pu-v}(\hat{\lambda}). \]
It is straightforward to check by induction on $a$ that the right-hand side of (5.10) is the $(u,v)$-entry in the matrix product
\[ C^I(\hat{\lambda}^{p^{a-1}})C^I(\hat{\lambda}^{p^{a-2}})\cdots C^I(\hat{\lambda}), \]
i.e., (5.10) implies the matrix congruence
\begin{equation}
B^I(\hat{\lambda})\equiv C^I(\hat{\lambda}^{p^{a-1}})C^I(\hat{\lambda}^{p^{a-2}})\cdots C^I(\hat{\lambda}) \pmod{pq^{\mu_I+1}}.
\end{equation}

We make explicit the matrix $C^I(\hat{\lambda})$.  Let $u,v\in U^I_{\min}$.  From (3.9) and (3.10) we have
\begin{equation}
\theta_{pu-v}(\hat{\lambda}) = \sum_{\substack{\nu\in{\mathbb N}^N\\ \sum_{j=1}^N \nu_j{\bf a}_j^+ = pu-v}} \bigg(\prod_{j=1}^N \theta_{\nu_j}\bigg)\hat{\lambda}^\nu.
\end{equation}
Since $u,v\in U^I_{\min}$, Lemma 2.1 implies that $\nu_j\leq p-1$ for all $j$ in the sum on the right-hand side of~(5.12).  It then follows from the definition of $\theta(t)$ that $\theta_{\nu_j} = \gamma_0^{\nu_j}/\nu_j!$.  Examining the last coordinate of the equation $\sum_{j=1}^N \nu_j{\bf a}_j^+ = pu-v$ gives $\sum_{j=1}^N \nu_j = (p-1)(\mu_I+1)$, so (5.12) can be simplified to
\begin{equation}
\theta_{pu-v}(\hat{\lambda}) = \gamma_0^{(p-1)(\mu_I+1)}\sum_{\substack{\nu\in{\mathbb N}^N\\ \sum_{j=1}^N \nu_j{\bf a}_j^+ = pu-v}} \frac{\hat{\lambda}^\nu}{\nu_1!\cdots\nu_N!}.
\end{equation}
Using (2.3), we obtain the relation between the matrices $C^I(\hat{\lambda})$ and~$A^I(\Lambda)$:
\begin{equation}
C^I(\hat{\lambda}) = (-\gamma_0^{p-1})^{\mu_I+1}A^I(\hat{\lambda}).
\end{equation}
From (5.11), we then obtain a relation between $A^I(\hat{\lambda})$ and $B^I(\hat{\lambda})$:
\begin{equation}
B^I(\hat{\lambda})\equiv (-\gamma_0^{p-1})^{a(\mu_I+1)} A^I(\hat{\lambda}^{p^{a-1}})A^I(\hat{\lambda}^{p^{a-2}})\cdots A^I(\hat{\lambda}) \pmod{pq^{\mu_I+1}}.
\end{equation}

Recall that $\sum_{i=0}^\infty \gamma_0^{p^i}/p^i = 0$ and that ${\rm ord}_p\:\gamma_0 = 1/p-1$.  In particular,
\[ {\rm ord}_p\: \frac{\gamma_0^{p^i}}{p^i} = \frac{p^i}{p-1} - i, \]
and since the right-hand side of this expression is an increasing function of $i$ for $i\geq 1$ we have
\[ \gamma_0 + \frac{\gamma_0^p}{p} \equiv 0 \pmod{\gamma_0 p^{p-1}}. \]
Multiplying by $p/\gamma_0$ then gives
\[ -\gamma_0^{p-1}\equiv p \pmod{p^p}. \]
It follows from this that
\[ (-\gamma_0^{p-1})^{a(\mu_I+1)}\equiv q^{\mu_I+1} \pmod{pq^{\mu_I+1}}, \]
so (5.15) may be simplified to
\begin{equation}
B^I(\hat{\lambda})\equiv q^{\mu_I+1} A^I(\hat{\lambda}^{p^{a-1}})A^I(\hat{\lambda}^{p^{a-2}})\cdots A^I(\hat{\lambda}) \pmod{pq^{\mu_I+1}}.
\end{equation}
Corollary 5.4 now implies the following congruences.
\begin{theorem}
If $n$ is not divisible by $d$, then
\begin{equation}
P_{\lambda}(q^{-\mu}t)\equiv\det\big(I-tA^S(\hat{\lambda}^{p^{a-1}})A^S(\hat{\lambda}^{p^{a-2}})\cdots A^S(\hat{\lambda})\big) \pmod{p}.
\end{equation}
If $n$ is divisible by $d$, then
\begin{multline}
P_{\lambda}(q^{-\mu}t)\equiv \\ 
\frac{\det\big(I-tA^S(\hat{\lambda}^{p^{a-1}})A^S(\hat{\lambda}^{p^{a-2}})\cdots A^S(\hat{\lambda})\big)}{\prod_{i=0}^n \det\big(I-tA^{S\setminus\{i\}}(\hat{\lambda}^{p^{a-1}})A^{S\setminus\{i\}}(\hat{\lambda}^{p^{a-2}})\cdots A^{S\setminus\{i\}}(\hat{\lambda})\big)}\pmod{p}.
\end{multline}
\end{theorem}

Since $A^S(\Lambda)$ is the matrix denoted $A(\Lambda)$ in the Introduction and since $\hat{\lambda}$ is the Teichm\"uller lifting of $\lambda$ we have $A^S(\hat{\lambda})\equiv \bar{A}^S(\lambda)\pmod{p}$.  Theorem~1.4 now follows from (5.18), and (5.19) is equivalent to
\begin{multline}
P_{\lambda}(q^{-\mu}t)\equiv \\ 
\frac{\det\big(I-t\bar{A}^S({\lambda}^{p^{a-1}})\bar{A}^S({\lambda}^{p^{a-2}})\cdots \bar{A}^S({\lambda})\big)}{\prod_{i=0}^n \det\big(I-t\bar{A}^{S\setminus\{i\}}({\lambda}^{p^{a-1}})\bar{A}^{S\setminus\{i\}}({\lambda}^{p^{a-2}})\cdots \bar{A}^{S\setminus\{i\}}({\lambda})\big)}\pmod{p}.
\end{multline}

\section{The case $d\,|\,n$}

We first give an example to show that Theorem 1.4 fails when $d\,|\, n$.  Consider the variety $X$ in ${\mathbb P}^n$ defined by the equation $x_1\cdots x_n=0$ (so we have $d=n$ and $\mu=1$).   A short calculation shows that its zeta function $Z(X,t)$ has the form (1.2) with
\[ P(t)^{(-1)^n} = \prod_{j=1}^{n-1} (1-q^jt)^{(-1)^{n-j}\binom{n-1}{j-1}}. \]
In particular, we have $R(q^{-1}t)\equiv (1-t)\pmod{q}$, contradicting Theorem~1.4.

An elementary observation will allow us to simplify the denominator of (5.20).  Assume that $d\,|\,n$ and $i\in\{0,1,\dots,n\}$.  Then
\[ \mu_{S\setminus\{i\}} + 1 = \bigg\lceil\frac{n}{d}\bigg\rceil = \mu. \]
This implies that $U^{S\setminus\{i\}}_{\min}$ is a singleton: $U^{S\setminus\{i\}}_{\min} = \{u^{(i)}\}$, where
\[ u^{(i)} = (1,\dots,1,0,1,\dots,1,\mu)\in{\mathbb N}^{n+2} \]
with zero in the $i$-th entry.  It follows that if we define a polynomial $g_i(\Lambda)\in({\mathbb Q}\cap{\mathbb Z}_p)[\Lambda]$ by the formula
\begin{equation}
g_i(\Lambda) = (-1)^\mu \sum_{\substack{\nu\in{\mathbb N}^N\\ \sum_{j=1}^N \nu_j{\bf a}_j^+ = (p-1)u^{(i)}}} \frac{\Lambda_1^{\nu_1}\cdots\Lambda_N^{\nu_N}}{\nu_1!\cdots\nu_N!}, 
\end{equation}
then by (2.3), $A^{S\setminus\{i\}}(\Lambda)$ is a one-by-one matrix with entry $g_i(\Lambda)$.  From (5.20) we then have the following result.
\begin{theorem}
If $d\,|\,n$, then
\begin{equation}
P_{\lambda}(q^{-\mu}t)\equiv 
\frac{\det\big(I-t\bar{A}^S({\lambda}^{p^{a-1}})\bar{A}^S({\lambda}^{p^{a-2}})\cdots \bar{A}^S({\lambda})\big)}{\prod_{i=0}^n \big(1-t\bar{g}_i({\lambda}^{p^{a-1}})\bar{g}_i({\lambda}^{p^{a-2}})\cdots \bar{g}_i({\lambda})\big)}\pmod{p}.
\end{equation}
\end{theorem}

\section{Generic invertibility of $\bar{A}^I(\Lambda)$}

The proof of generic invertibility follows the lines of our recent proof of generic invertibility for the Hasse-Witt matrix\cite{AS1}.  We fix $I$ and prove that the matrix $\bar{A}^I(\Lambda)$ is generically invertible, in a sense made precise below.  

We consider the following condition on the set $\{{\bf a}_j\}_{j=1}^N$.  Suppose that $N\geq \mu_I + |U^I_{\min}|$ and that the elements $\{ {\bf a}_j \}_{j=1}^{\mu_I}$ and $\{ {\bf a}_j\}_{j=\mu_I+1}^{\mu_I+|U^I_{\min}|}$ have the following property:
\begin{hypothesis}
For each $u\in U^I_{\min}$, there exists (a necessarily unique) $k_u$, $1\leq k_u\leq |U^I_{\min}|$, such that $x^u=x^{{\bf a}_{\mu_I+k_u}^+}  \prod_{j=1}^{\mu_I} x^{{\bf a}_j^+}$.
\end{hypothesis}

We show that such subsets always exists when, for example, $\{x^{{\bf a}_j}\}_{j=1}^N$ consists of all monomials of degree $d$.  To fix ideas and simplify notation, suppose that $|I|=h$, $0\leq h\leq n+1$ and that $I=\{0,1,\dots,h-1\}$.  
For $j=1,\dots,\mu_I$ we take $x^{{\bf a}_j} = \prod_{k=(j-1)d}^{jd-1} x_k$
and we take $\{ x^{{\bf a}_j}\}_{j=\mu_I+1}^{\mu_I+|U^I_{\min}|}$ to consist of all monomials of degree~$d$ that are divisible by the product $x_{\mu_Id}\cdots x_{h-1}$.  Then Hypothesis~7.1 is satisfied.  

\begin{theorem}
If $\{{\bf a}_j\}_{j=1}^N$ satisfies Hypothesis $7.1$, then the matrix $\bar{A}^I(\Lambda)$ is invertible.
\end{theorem}

We begin with a reduction step.  Define a related matrix $D^I(\Lambda) = [D^I_{uv}(\Lambda)]_{u,v\in U^I_{\min}}$ by setting
\begin{multline*}
 D^I_{uv}(\Lambda) = \bigg(\prod_{j=1}^{\mu_I} \Lambda_j\bigg)^{-(p-1)} \Lambda_{\mu_I+k_u}^{-p} \Lambda_{\mu_I+k_v} A^I_{uv}(\Lambda) \in \\ 
\Lambda_{\mu_I+k_u}^{-1}\Lambda_{\mu_I+k_v}({\mathbb Q}\cap{\mathbb Z}_p)[\Lambda_1^{-1},\dots,\Lambda_{\mu_I}^{-1},\Lambda_{\mu_I+k_u}^{-1},\Lambda_{\mu_I+2},\dots,\hat{\Lambda}_{\mu_I+k_u},\dots,\Lambda_N] ,
\end{multline*}
where $A^I_{uv}(\Lambda)$ is given by (2.3).  In other words, we obtain $D^I(\Lambda)$ by multiplying row $u$ of $A^I(\Lambda)$ by $\big(\Lambda_{\mu_I+k_u}\prod_{j=1}^{\mu_I}\Lambda_j\big)^{-p}$ and multiplying column $v$ of $A^I(\Lambda)$ by $\Lambda_{\mu_I+k_v}\prod_{j=1}^{\mu_I}\Lambda_j$.  It follows that
\begin{equation}
 \det D^I(\Lambda) = \bigg(\prod_{j=1}^{\mu_I} \Lambda_j\bigg)^{-(p-1)|U^I_{\min}|}\bigg(\prod_{j=1}^{|U^I_{\min}|} \Lambda_{\mu_I + j}\bigg)^{-(p-1)}\det A^I(\Lambda),
\end{equation}
hence to prove Theorem 7.2 it suffices to prove that 
\begin{equation}
\det \bar{D}^I(\Lambda)\neq 0.
\end{equation}

For $u\in U^I_{\min}$, put
\begin{multline*}
L_u = \bigg\{ l=(l_1,\dots,l_N)\in{\mathbb Z}^N \mid \text{$\sum_{k=1}^N l_k{\bf a}_k^+ = {\bf 0}$,}\\ \text{$l_j\leq 0$ for $j=1,\dots,\mu_I$ and $j=\mu_I+k_u$, $l_j\geq 0$ otherwise}\bigg\}. 
\end{multline*}

\begin{lemma}
If the monomial  $\Lambda^l$ appears in the Laurent polynomial $D^I_{uv}(\Lambda)$, then $l\in L_u$.
\end{lemma}

\begin{proof}
It follows from Hypothesis 7.1 that
\[ u={\bf a}_{\mu_I+k_u}^+ + \sum_{j=1}^{\mu_I}{\bf a}_j^+\quad\text{and}\quad v={\bf a}_{\mu_I+k_v}^+ + \sum_{j=1}^{\mu_I}{\bf a}_j^+. \]
The formula for $A^I_{uv}(\Lambda)$ is given in (2.3), where each exponent $\nu_j$ satisfies $0\leq\nu_j\leq p-1$ by Lemma~2.1.  The assertion of Lemma 7.5 now follows from the definition of $D^I_{uv}(\Lambda)$.
\end{proof}

\begin{lemma}
The Laurent polynomial $D^I_{uv}(\Lambda)$ has no constant term if $u\neq v$ and has constant term equal to $\big(-(p-1)!\big)^{-(\mu_I+1)}$ if $u=v$.
\end{lemma}

\begin{proof}
If $u\neq v$, the definition of $D^I_{uv}(\Lambda)$ shows that every monomial in $D^I_{uv}(\Lambda)$ contains a negative power of $\Lambda_{\mu_I + k_u}$ and a positive power of $\Lambda_{\mu_I+k_v}$.  If $u=v$, then
\[ pu-v = (p-1)u = (p-1)\bigg({\bf a}^+_{\mu_I+k_u} + \sum_{j=1}^{\mu_I}{\bf a}_j^+\bigg), \]
so (2.3) shows that the monomial
\[ (-1)^{\mu_I+1} \frac{\big(\Lambda_{\mu_I+k_u}\prod_{j=1}^{\mu_I}\Lambda_j\big)^{p-1}}{(p-1)!^{\mu_I+1}} \]
appears in $A^I_{uv}(\Lambda)$.  The assertion of the lemma then follows from the definition of $D^I_{uv}(\Lambda)$.
\end{proof}

We shall prove (7.4) by showing that $\det \bar{D}^I(\Lambda)$ has a nonzero constant term.  In fact, we shall show that the constant term of $\det D^I(\Lambda)$ is a $p$-adic unit.  By Lemma~7.5 and the following proposition, the constant term of $\det D^I(\Lambda)$ is the determinant of the matrix whose $(u,v)$-entry is the constant term of $D^I_{uv}(\Lambda)$.  And by Lemma~7.6, this matrix of constant terms is a diagonal matrix whose diagonal entries are $p$-adic units.  Thus the following proposition completes the proof of~(7.4).

\begin{proposition}
Let $l^{(u)} = (l^{(u)}_1,\dots,l^{(u)}_N)\in L_u$ for $u\in U^I_{\min}$.  One has 
\begin{equation}
\sum_{u\in U^I_{\min}} l^{(u)} = {\bf 0} 
\end{equation}
if and only if $l^{(u)} = {\bf 0}$ for all $u\in U^I_{\min}$.
\end{proposition}

\begin{proof}
The ``if'' part of the proposition is clear, so consider a set $\{l^{(u)}\}_{u\in U^I_{\min}}$ with $l^{(u)}\in L_u$ satisfying (7.8).  Since $l^{(u)}\in L_u$ we have
\begin{equation}
\sum_{k=1}^N l^{(u)}_k {\bf a}_k^+ = {\bf 0},
\end{equation}
which  implies (since the last coordinate of each ${\bf a}_k^+$ equals 1)
\begin{equation}
\sum_{k=1}^N l^{(u)}_k = 0.
\end{equation}
From the definition of $L_u$, we have $l^{(u)}_j\leq 0$ for all $u$ if $j\leq \mu_I$ and $l^{(u)}_j\geq 0$ for all $u$ if $j\geq \mu_I+|U^I_{\min}|+1$, so (7.8) implies that
\begin{equation}
\text{$l^{(u)}_j = 0$ for all $u$ and all $j\leq\mu_I$}
\end{equation}
and
\begin{equation}
\text{$l^{(u)}_j = 0$ for all $u$ and all $j\geq\mu_I+|U^I_{\min}|+1$.}
\end{equation}
To establish the proposition, it remains to show that $l^{(u)}_j=0$ for all $u$ if $\mu_I+1\leq j\leq \mu_I+|U^I_{\min}|$.  

Using (7.11) and (7.12), Equations (7.9) and (7.10) become
\begin{equation}
\sum_{k=1}^{|U^I_{\min}|} l^{(u)}_{\mu_I+k} {\bf a}_k^+ = {\bf 0}
\end{equation}
and
\begin{equation}
\sum_{k=1}^{|U^I_{\min}|} l^{(u)}_{\mu_I+k} = 0
\end{equation}
for all $u$.  Since $l^{(u)}\in L_u$, we have $l^{(u)}_{\mu_I+k_u}\leq 0$ and $l^{(u)}_{\mu_I+k}\geq 0$ if $k\neq k_u$.  If $l^{(u)}_{\mu_I+k_u} = 0$, then (7.14) implies that $l^{(u)}_{\mu_I+k} = 0$ for $k\neq k_u$ as well, so $l^{(u)} = {\bf 0}$.

If follows that if $l^{(u)}\neq{\bf 0}$, then $l^{(u)}_{\mu_I+k_u}<0$ so we can solve (7.13) for~${\bf a}_{k_u}^+$:
\begin{equation}
{\bf a}_{k_u}^+ = \sum_{\substack{k=1\\ k\neq k_u}}^{|U^I_{\min}|}\bigg( -\frac{l^{(u)}_{\mu_I+k}}{l^{(u)}_{\mu_I+k_u}}\bigg) {\bf a}_k^+.
\end{equation}
The coefficients on the right-hand side of (7.15) are nonnegative rational numbers that sum to 1, so (7.15) implies the following statement.
\begin{lemma}
If $l^{(u)}\neq{\bf 0}$, then ${\bf a}_{k_u}^+$ is an interior point of the convex hull of the set $\{ {\bf a}_k^+\mid l^{(u)}_{\mu_I+k}>0\}$.
\end{lemma}

Let $Z=\{ {\bf a}_k^+\mid \text{$l^{(u)}_{\mu_I+k}\neq 0$ for some $u\in U^I_{\min}$}\}$.  If ${\bf a}_k^+\not\in Z$, then $l^{(u)}_{\mu_I+k}=0$ for all $u\in U^I_{\min}$, so if $Z=\emptyset$ then $l^{(u)}_{\mu_I+k} = 0$ for all $k$, $1\leq k\leq |U^I_{\min}|$, and all $u\in U^I_{\min}$, which establishes the proposition.  So suppose $Z\neq\emptyset$ and choose $k_0$, $1\leq k_0\leq |U^I_{\min}|$, such that ${\bf a}_{k_0}^+$ is a vertex of the convex hull of $Z$.    If $l^{(u)}\neq{\bf 0}$, then Lemma 7.16 implies that $k_0\neq k_u$.  By the definition of $L_u$, we therefore have $l^{(u)}_{\mu_I+k_0}\geq 0$ for all $u\in U^I_{\min}$.  Furthermore, since ${\bf a}_{k_0}^+\in Z$, we must have $l^{(u)}_{\mu_I+k_0}>0$ for some $u\in U^I_{\min}$.  It follows that $\sum_{u\in U^I_{\min}} l^{(u)}_{\mu_I+k_0}>0$, contradicting (7.8).  Thus we must have $Z=\emptyset$.
\end{proof}


\begin{thebibliography}{99}

\bibitem{AS} A. Adolphson and S. Sperber.  Hasse invariants and ${\rm mod}\, p$ solutions of $A$-hypergeometric systems.  J. Number Theory {\bf 142} (2014), 183--210. 

\bibitem{AS1} A. Adolphson and S. Sperber.  $A$-hypergeometric series and the Hasse-Witt matrix of a hypersurface.  Finite Fields and Their Applications {\bf 41} (2016), 55--63. 

\bibitem{AS2} A. Adolphson and S. Sperber.  Distinguished-root formulas for generalized Calabi-Yau hypersurfaces.  Preprint, available at {\tt arXiv:1602.03578}.

\bibitem{A} J. Ax.  Zeroes of polynomials over finite fields.  Amer.\ J. Math.\ {\bf 86} (1964), 255--261.

\bibitem{K} N. Katz.  On a theorem of Ax.  Amer.\ J. Math.\  {\bf 93}  (1971), 485--499.

\bibitem{K2} N. Katz.  Algebraic solutions of differential equations ($p$-curvature and the Hodge filtration).  Invent.\ 
Math.\ {\bf 18} (1972), 1--118.

\bibitem{K3} N. Katz.  Une formule de congruence pour la fonction $\zeta$.  Expos\'e XXII in S.~G.~A.~7~II, Lecture Notes in Mathematics 340, Springer, Berlin-Heidelberg-New York, 1973.  

\bibitem{S} J.-P. Serre.  Endomorphismes compl\`etement continus des espaces de Banach $p$-adiques. Inst.\ Hautes \'Etudes Sci.\ Publ.\ Math.\ No.\ 12 (1962), 69--85.


\end{thebibliography}
\end{document}